\pdfoutput=1
\documentclass{article}
\usepackage{amsthm}
\usepackage{mlapa}
\usepackage{algorithmic}
\usepackage{zshmacros}

\begin{document}
\title{A short note on the tail bound of Wishart distribution}
\author{Shenghuo Zhu\\ \texttt{zsh@nec-labs.com}}
\maketitle

\begin{abstract}
  We study the tail bound of the emperical covariance of multivariate
  normal distribution. Following the work of \cite{gittens11:_tail},
  we provide a tail bound with a small constant.
\end{abstract}

\section{Main result}
Let $\{\xi_k: k=1\cdots n\}$ follow multivariate normal distribution
$\calN_d(0,C)$. The scatter matrix $S=\sum_{k=1}^n \xi_k
\trans{\xi_k}$ follows Wishart distribution, $\calW_d(n,C)$. The
estimate of $C$ is $\frac{1}{n} S$. The tail bound of $S$ has a wide
range of applications, such as, the sample estimation of random
projection. We follow the work of \cite{gittens11:_tail} to find the
tail bound with smaller constants.

\noindent\textbf{Notation:} Let denote the $\ell$-th largest
eigenvalue of matrix $X$ by $\lambda_\ell(X)$, the trace of $X$ by
  $\tr(X)$, and the spectral norm of $X$ by $\|X\|$.

\begin{thm}\label{thm:1}
  If $S$ follows a Wishart distribution $\calW_d(n,C)$, then for
  $\theta \geq 0$,
  \begin{align}
    \Pr\left\{\lambda_1(\frac{1}{n}S-C) \geq
      \left(\sqrt{\frac{2\theta (r+1)}{n}} + \frac{2 \theta
          r}{n}\right)\lambda_1(C)\right\}
    &\leq d\exp\{-\theta\}, \label{eq:15} \\
    \Pr\left\{\lambda_1(C-\frac{1}{n}S) \geq
      \left(\sqrt{\frac{2\theta (r+1)}{n}} + \frac{2 \theta
          r}{n}\right)
      \lambda_1(C)\right\} & \leq d\exp\{-\theta\}, \label{eq:16} \\
    \Pr\left\{\|\frac{1}{n}S-C\| \geq
      \left(\sqrt{\frac{2\theta(r+1)}{n}} + \frac{2 \theta r
        }{n}\right)\|C\|\right\} &\leq
    2d\exp\{-\theta\}, \label{eq:17} \\
    \Pr\left\{|\lambda_\ell(\frac{1}{n}S)-\lambda_\ell(C)| \geq
      \left(\sqrt{\frac{2\theta \kappa_\ell^2 (r+1)}{n}} + \frac{2
          \theta \kappa_\ell r}{n}\right)\lambda_\ell(C), \forall
      \ell\in\{1\cdots d\}\right\} &\leq
    2d\exp\{-\theta\}\label{eq:18},
  \end{align}
  where $r=\tr(C)/\|C\|$, and condition numbers
  $\kappa_\ell=\lambda_1(C)/\lambda_\ell(C)$.
\end{thm}

\begin{remark}
  When $d=1$ and $C=1$, then $r=1$, and it is exactly the upper bound
  of chi-square distribtuion provided in
  in~\cite{laurent00:_adapt}.
\end{remark}

\begin{remark} Applying the modification in this note to Theorem 7.1
  of \cite{gittens11:_tail}, we have
  \begin{align}
    \Pr\left\{\lambda_\ell(\frac{1}{n}S) \geq
      \left(1+\sqrt{\frac{2\theta (\kappa_\ell r_\ell+2)}{n}} +
        \frac{2 \theta \kappa_\ell r_\ell}{n}
      \right)\lambda_\ell(C)\right\} &\leq (d-\ell+1)
    \exp\{-\theta\}, \text{~for $\ell=1\cdots d$}, \label{eq:19}\\
    \Pr\left\{\lambda_\ell(\frac{1}{n}S) \leq
      \left(1-\sqrt{\frac{2\theta \kappa_\ell^2
            (r_1-r_{\ell+1}+2)}{n}}\right)\lambda_\ell(C)\right\}
    &\leq \ell \exp\{-\theta\}, \text{~for $\ell=1\cdots
      d$},\label{eq:20}
  \end{align}
  where $r_\ell=\sum_{i=\ell}^d \lambda_i(C)/\lambda_1(C)$. As
  $r_\ell$ is smaller than $r$, it is tighter
  individually. Eq.~(\ref{eq:19}) and (\ref{eq:20}) are individual
  eigenvalue bounds, but Eq.~(\ref{eq:18}) is the collective
  eigenvalue bound. When $\ell=1$, $\kappa_1=1$ and $r_\ell=r$, then
  the upper bound of the top eigenvalue of Eq.~(\ref{eq:19}) is
  slightly looser than that of Eq.~(\ref{eq:15}).
\end{remark}

\section{Proof}
We use part of the proof of Lemma 8 in \cite{birg98:_minim}.
\begin{lem}\label{lem:1}
  Let $B>0$ and $\sigma>0$.  If the log-moment generating function
  satisfies
  \begin{align*}
    \log \expect \exp\{u Z\} &\leq \frac{\sigma^2 u^2}{2 (1-uB)}
    \quad \text{for all $0 \leq u < 1/B$}, 
  \end{align*}
  then
  \begin{equation}
    \label{eq:1}
    \Pr\{ Z \geq \epsilon\} \leq
    \exp\{-\frac{\epsilon^2}{2 \sigma^2 + 2 \epsilon B}\}
    \quad\text{for all $\epsilon \geq 0$},
  \end{equation}
  and
  \begin{equation}
    \label{eq:2}
    \Pr\{ Z \geq \sqrt{2\theta \sigma^2} + \theta B\} \leq
    \exp\{-\theta\}
    \quad \text{for all $\theta \geq 0$}.
  \end{equation}
\end{lem}
\begin{proof} It follows Markov's inequality that
  \begin{align*}
    \Pr\{ Z \geq \epsilon\} & \leq \inf_u \expect \exp\{ -u \epsilon +
    u Z\} = \exp\{-h(\epsilon)\},
  \end{align*}
  where $h(\epsilon):=\sup_{u} u\epsilon - \frac{\sigma^2 u^2}{2
    (1-u B)}$.  Also, the supremum is achieved for
  \begin{equation*}
    \epsilon= \frac{\sigma^2u}{1-u B}+
    \frac{\sigma^2u^2 B}{2(1-u B)^2}=\frac{\sigma^2u}{2(1-uB)}+
    \frac{\sigma^2u}{2(1-u B)^2},
  \end{equation*}
  i.e. $u=B^{-1}[1-\sigma(2\epsilon B + \sigma^2)^{-1/2}] <
  1/B$. Then we prove Eq.~(\ref{eq:1}), as
  \begin{equation*}
    h(\epsilon)=\frac{\epsilon^2}{\epsilon B
      +\sigma^2+\sigma^2(1+2\epsilon
      B/\sigma^2)^{1/2}} \geq \frac{\epsilon^2}{2 \epsilon B + 2\sigma^2}.
  \end{equation*}
  Let
  \begin{equation*}
    \theta:=\frac{\sigma^2u^2}{2(1-u B)^2}
    =h(\epsilon).
  \end{equation*}
   Then we prove Eq.~(\ref{eq:2}), as
   \begin{align*}
    \sqrt{2\theta \sigma^2}+\theta B = \frac{\sigma^2u}{(1-u B)}+
    \frac{\sigma^2u^2 B}{2(1-u B)^2} = \epsilon.
  \end{align*}
\end{proof}

The following Theorem is Theorem 6.2 in~\cite{tropp:_user}, except for
using Lemma~\ref{lem:1} to achieve a different formula.
\begin{thm}\label{thm:m}
  If a finite sequence $\{X_k: k=1\cdots n\}$ of independent, random,
  self-adjoint matrices with dimension $d$, all of which satisfy the
  Bernstein's moment condition, i.e.
  \begin{equation*}
    \expect X_k^p \preceq \frac{p!}{2} B^{p-2} \Sigma_2, \text{~for~}
    p\geq 2,
  \end{equation*}
  where $B$ is a positive constant and $\Sigma_2$ is a positive
  semi-definite matrix, then, 
  \begin{align*}
    \log \expect \exp (u X_k) & \preceq u \expect X_k +
    \frac{u^2}{2(1- u B)} \Sigma_2 \quad \text{for all $0\leq u < 1/B$,}\\
    \Pr\{ \lambda_1(\sum_k X_k) &\geq \lambda_1(\sum_k \expect X_k)+
    \sqrt{2{n}\theta \lambda_1(\Sigma_2) } + \theta B \} \leq d
    \exp\{-\theta\}.
  \end{align*}
  Additionally, if $X_k$ are positive semi-definite matrices,
  \begin{align*}
    \log \expect \exp (-u X_k) & \preceq -u \expect X_k +
    \frac{u^2}{2} \Sigma_2 \quad
    \text{for all $u\geq 0$,} \\
    \Pr\{ \lambda_d(\sum_k X_k) &\leq \lambda_d(\sum_k \expect X_k)-
    \sqrt{2\theta n\lambda_1(\Sigma_2)} \} \leq d \exp\{-\theta\}.
  \end{align*}
\end{thm}
\begin{proof}
  \begin{align*}
    \log \expect \exp (u X_k) & =
    \log(I + u \expect X_k + \sum_{p=2}^\infty \frac{u^p}{p!} \expect X_k^p) \\
    & \preceq u \expect X_k + \sum_{p=2}^\infty \frac{u^2 (u
      B)^{p-2}}{2}
    \Sigma_2 \\
    &= u \expect X_k + \frac{u^2}{2(1-uB)} \Sigma_2.
  \end{align*}
  It follows Theorem 3.6 in \cite{tropp:_user} and Lemma \ref{lem:1}
  that
  \begin{align*}
    \Pr\{\lambda_1(\sum_k X_k) \geq
    \lambda_1(\sum_k \expect X_k)+ \epsilon \} & \leq
    \inf_{u \geq 0} \left\{ \exp( -u \lambda_1(\sum_k
      \expect X_k) -u \epsilon) \tr\exp(\sum_k \log \expect \exp(u
      X_k))\right\} \\
    & \leq \inf_{u \geq 0} \left\{ d \exp(-u \epsilon + \frac{n
        u^2}{2(1-uB)} \lambda_1(\Sigma_2))\right\} \\
    & \leq d \exp(-\theta),
  \end{align*}
  where $\epsilon= \sqrt{2n\theta \lambda_1(\Sigma_2)}
  + \theta B $.
  \begin{align*}
    \log \expect \exp (-u X_k) & \preceq
    \log(I - u \expect X_k + \frac{u^2}{2} \expect X_k^2) \\
    & \preceq -u \expect X_k + \frac{u^2}{2} \Sigma_2 ,
  \end{align*}
  then
  \begin{align*}
    \Pr\{\lambda_d(\sum_k X_k) \leq \lambda_d(\sum_k \expect X_k)-
    \epsilon \} & \leq \inf_{u \geq 0} \left\{ \exp(u
      \lambda_d(\sum_k \expect X_k) -u \epsilon) \tr\exp(\sum_k
      \log \expect \exp(-u
      X_k))\right\} \\
    & \leq \inf_{u \geq 0} \left\{ d \exp(-u \epsilon + \frac{n
        u^2}{2} \lambda_1(\Sigma_2))\right\} \\
    & \leq d \exp(-\theta),
  \end{align*}
  where $\epsilon=\sqrt{2\theta n \lambda_1(\Sigma_2)}$.
\end{proof}

Then we prove the Bernstein's moment condition for $\xi\trans{\xi}$
and $\xi\trans{\xi}-C$.
\begin{lem}\label{lem:3}
  Let $\xi$ be random vectors from $\calN_d(0,C)$. For $p\geq 2$,
  \begin{align*}
    \expect (\xi\trans{\xi})^p \preceq \frac{p!}{2}
    B^{p-2}(\tr(C)C+2C^2), \\
    \expect (\xi\trans{\xi}-C)^p \preceq \frac{p!}{2}
    B^{p-2} \Sigma_2, \\
    \expect (C-\xi\trans{\xi})^p \preceq \frac{p!}{2} B^{p-2}
    \Sigma_2,
  \end{align*}
  where $\Sigma_2=\tr(C)C+C^2$ and $B= 2\tr(C)$.
\end{lem}
\begin{proof}
  Let $X=\xi\trans{\xi}$ and $\Sigma_p=\expect (X-C)^p$, for $p\geq 2$.
  It follows Isserlis' theorem~\cite{isserlis18} that
  \begin{align*}
    (\expect X^2)_{ij}&=\sum_{k}\expect\xi_i \xi_k^2\xi_j =[\expect
    \xi_i\xi_j][\sum_k \expect \xi_k^2] +2\sum_k [\expect
    \xi_{ik}][\expect \xi_{jk}] =\tr(C)C_{ij}+2(C^2)_{ij}, \\
    \Sigma_2&=\expect X^2 -C^2 = \tr(C)C+C^2.
  \end{align*}
  Then, we calculate $\expect X^3$ and $\Sigma_3$ to get the basic
  idea.
  \begin{align*}
    \expect X^3 &=\tr(C)^2C+2\tr(C^2)C+4\tr(C)C^2+8C^3 \preceq
    5\tr(C)(\tr(C)C+2C^2)
    \preceq \frac{3!}{2} B^{3-1}(\tr(C)C+2C^2),\\
    \Sigma_3&=\expect X^3 - \expect XCX - \expect X^2C - \expect CX^2
    + \expect C^2X + \expect CXC + \expect XC^2
    - C^3  \\
    &= \expect X^3 - \expect XCX -2(\tr(C) C^2+C^3) \\
    &= (\tr(C)^2C+2\tr(C^2)C+4\tr(C)C^2+8C^3) - (\tr(C^2)C+2 C^3) -2 (\tr(C) C^2+C^3) \\
    &= \tr(C)^2C+\tr(C^2)C+2\tr(C)C^2+4C^3 \preceq 4
    \tr(C)(\tr(C)C+C^2)
    \preceq \frac{3!}{2} B^{3-1}\Sigma_2, \\
    \expect (C-X)^3 &= -\Sigma_3 \preceq 0 \preceq
    \frac{3!}{2} B^{3-1}\Sigma_2. 
  \end{align*}
  Let $Z_{k,i}=\prod_j Y_{k,i,j}$, where $Y_{k,i,j}$ is $X$ or $C$,
  $k$ is the number of $C$'s in the term between $0$ and $p$, $i$ is
  the index term between $1$ and ${p \choose k}$, and $j$ is between
  $1$ and $p$. Each element of $Y_{k,i,j}$ can be written as
  $\xi_{l_{j-1}}\xi_{l_{j}}$ or $C_{l_{j-1},l_{j}}$, where $l_j$ is
  between $1$ and $d$.  It follows Isserlis' theorem that the
  expectation of each element $\expect Z_{k,i}$ is the sum of the
  product of the expectations of $\xi_l \xi_{l'}$ all
  combinations. For example, in $p=3$, we write $Z_{1,2}=XCX$, then
  \begin{align}
    \expect Z_{1,2} &=
    \begin{pmatrix}
    \expect \sum_{l_1,\cdots,l_{2}} \xi_{l_0} \xi_{l_1} C_{l_1,l_2}
    \xi_{l_2} \xi_{l_3} 
    : l_0, l_3 \in \{1\cdots d\}
    \end{pmatrix}
    \nonumber\\
    &= 
    \begin{pmatrix}
    \sum_{l_1,l_{2}} [\expect(\xi_{l_0}
      \xi_{l_1})C_{l_1,l_2}\expect(\xi_{l_2} \xi_{l_3})
      +\expect(\xi_{l_0} \xi_{l_2})C_{l_1,l_2}\expect(\xi_{l_1}
      \xi_{l_3}) +\expect(\xi_{l_0}
      \xi_{l_3})C_{l_1,l_2}\expect(\xi_{l_1} \xi_{l_2})]
    : l_0, l_3 \in \{1\cdots d\}
    \end{pmatrix}
    \nonumber \\
    &= 
    \begin{pmatrix}
    \sum_{l_1,l_{2}} [C_{l_0,l_1}C_{l_1,l_2}C_{l_2,l_3}
      +C_{l_0,l_2}C_{l_1,l_2}C_{l_1,l_3}
      +C_{l_0,l_3}C_{l_1,l_2}C_{l_1,l_2}
      ]
    : l_0, l_3 \in \{1\cdots d\}
    \end{pmatrix}
    \nonumber \\
    &= [(0 1) (1 2) (2 3)] +[(0 2) (1 2) (1 3)] +[(0 3) (1 2) (1 2)],
    \label{eq:3} \\
    &= [(0 1 2 3)] +[(0 2 1 3)] +[(0 3) (1 2 1)] \label{eq:4} \\
    &= C^3 +C^3 + \tr(C^2) C\label{eq:5}
  \end{align}
  In Eq~(\ref{eq:3}), each $C$ is written a pair, and each product as
  a list.  In Eq~(\ref{eq:4}), pairs are combined into one chain and
  serveral loops. Then in Eq~(\ref{eq:5}), each chain is $C^{c}$,
  where $c$ is the lenth of the chain, and each loop is $\tr(C^l)$,
  where $l$ is the length of the loop.  In general, 
  $\expect Z_{k,i}$ is the sum of terms like $C^{c} \prod_j \tr(C^{l_{j}})$.

  We have $C^c \preceq \tr(C)^{c-2} C^2 \preceq \tr(C)^{c-1}C$, and
  $\tr(C^l)\leq \tr(C)^l$, so we only count the terms with singleton
  chain, i.e. $c=1$, and all terms to bound the expectations with
  $\tr(C)^{p-2}(\tr(C)C+ 2C^2)$ or $\tr(C)^{p-2}(\tr(C)C+
  C^2)$. $\expect Z_{k,i}$ is a expectation of $(2p-2k)$-order
  moments, which yields $(2p-2k-1)!!$ terms. For a given $k$, we have
  ${p\choose k} (2p-2k-1)!!$ terms, assuming $(-1)!!=1$.  A singleton
  chain term must contain $(0,p)$, thus $Z_{k,i}$ must be $X
  (\prod_{j=2}^{p-1} Y_{k,i,j})X$. For a given $k$, the number of
  singleton chain terms is ${p-2 \choose k} (2p-2k-3)!!$.

  For $\expect X^p=\expect Z_{0,1}$ has $(2p-1)!!$ terms, which
  include $(2p-3)!!$ singleton chain terms. The number of singleton
  chain terms is less than a third of the number of all terms when
  $p\geq 2$.  For $p\geq 2$,
  \begin{align*}
    \expect X^p & \preceq \frac{(2p-1)!!}{3} (\tr(C)^{p-1} C
    +2\tr(C)^{p-2}C^2) = \frac{\Gamma(p+1/2)}{3\sqrt{\pi} \Gamma(p+1)}
    p!2^p
    (\tr(C)^{p-1} C +2\tr(C)^{p-2}C^2) \\
    &\preceq \frac{1}{8} p!  2^p (\tr(C)^{p-1} C +2\tr(C)^{p-2}C^2) =
    \frac{p!}{2} B^{p-2} (\tr(C)C +2C^2).
  \end{align*}

  Then $\expect (X-C)^p=\sum_k (-1)^k \sum_i Z_{k,i}$.  The number of
  singleton chain terms is less than half of the number of all
  terms. Thus
  \begin{align*}
    \Sigma_4 &\preceq 10 \tr(C)^{4-1} C + 50\tr(C)^{4-2} C^2 \preceq
    30 \tr(C)^{4-1} C + 30\tr(C)^{4-2} C^2
    \preceq \frac{4!}{2} B^{4-2}\Sigma_2, \\
    \expect (C-X)^4 &= \Sigma_4
    \preceq \frac{4!}{2} B^{4-2}\Sigma_2. \\
  \end{align*}
  When $p\geq 5$,
  \begin{align*}
    \Sigma_p &\preceq \expect X^p + C^p \preceq
    \frac{(2p-1)!!+1}{2}(\tr(C)^{p-1} C+\tr(C)^{p-2} C^2) \\
    &=(\frac{\Gamma(p+1/2)}{2\sqrt{\pi} \Gamma(p+1)} +
    \frac{1}{p!2^{p+1}})
    p! 2^p \tr(C)^{p-2}(\tr(C)C +C^2)\\
    &\preceq 0.1232 \times p! 2^p \tr(C)^{p-2}(\tr(C)C +C^2) 
    \preceq
    \frac{p!}{2} B^{p-2} \Sigma_2,\\
    \expect (C-X)^p &
    \preceq \expect X^p + C^p \preceq
    \frac{p!}{2} B^{p-2} \Sigma_2.\\
  \end{align*}
\end{proof}

Now we prove  Theorem~\ref{thm:1}.
\begin{proof}[Proof of Theorem~\ref{thm:1}]
  Let $X_k=\xi_k\trans{\xi_k}-C$. We have $\expect X_k=0$,
  $\lambda_1(\Sigma_2)\leq (r+1)\lambda_1(C)^2$, and $B=
  2r\lambda_1(C)$. Then Eq~(\ref{eq:15}) follows Lemma~\ref{lem:3} and
  Theorem~\ref{thm:m}. Similarly, letting $X_k=C-\xi_k\trans{\xi_k}$,
  we prove Eq~(\ref{eq:16}).  Combining them and $\|C\|=\lambda_1(C)$,
  we have Eq~(\ref{eq:17}). Plugging $\lambda_1(C)=\kappa_\ell
  \lambda_\ell(C)$, Eq~(\ref{eq:18}) follows Weyl's theorem on
  eigenvalues, specifically,
  \begin{align*}
    \lambda_\ell(\frac{1}{n}S) &\leq \lambda_\ell(C)+
    \lambda_1(\frac{1}{n}S-C), \\
    \lambda_\ell(C) &\leq \lambda_\ell(\frac{1}{n}S) +
    \lambda_1(C-\frac{1}{n}S).
  \end{align*}
\end{proof}
\bibliographystyle{mlapa}
\bibliography{zsh}

\end{document}